\documentclass[12pt]{article} 
\usepackage{amsmath, amssymb, mathrsfs, tensor}

\usepackage{xcolor, graphicx, subfigure}

\numberwithin{equation}{section}

\usepackage{amsthm}
\newtheorem{theorem}{Theorem}[section]
\newtheorem{proposition}[theorem]{Proposition}
\newtheorem{lemma}[theorem]{Lemma} 

\evensidemargin \oddsidemargin

\begin{document}

\vspace{5mm}

\begin{center}

% Title
{\Large \bfseries
	Existence and Uniqueness of BPS Vacuum and Multi-vortices in Inhomogeneous Abelian Higgs Model
}\\[17mm]

% Authors
	SeungJun Jeon\textsuperscript{1},
	~~Chanju Kim\textsuperscript{2},
	~~Yoonbai Kim\textsuperscript{1}
\\[2mm]
{\itshape
	\textsuperscript{1}
	Department of Physics,
    Sungkyunkwan University,
    Suwon 16419,
    Korea
}
\\[-1mm]
{\itshape
	\textsuperscript{2}
	Department of Physics,
	Ewha Womans University,
	Seoul 03760,
	Korea
}
\\[-1mm]
{\itshape
	sjjeon@skku.edu,
	~cjkim@ewha.ac.kr,
	~yoonbai@skku.edu
}

\end{center}

\vspace{15mm}

\begin{abstract}

\noindent
The BPS limit of the inhomogeneous abelian Higgs model is considered in \((1+2)\)-dimensions.
The second order Bogomolny equation is examined in the presence of an inhomogeneity expressed as a function of spatial coordinates.
Assuming a physically reasonable upper bound on the \(L^{2}(\mathbb{R}^{2})\) norm of the inhomogeneity function, we prove the existence and the uniqueness of nontrivial BPS vacuum solution of zero energy and topological BPS multi-vortex solutions of quantized positive energies.

\end{abstract}

\clearpage

\section{Introduction}

In numerous studies on solitons in nonlinear field theories, BPS objects have been a powerful tool because of their tractability.
Among various field theoretic models including the BPS limit, abelian Higgs model has played an outstanding role.
The nonrelativistic version, the Ginzburg-Landau model, describes normal superconductivity as a macroscopic model \cite{Cohen:2016, Arovas:2019} and the Abrikosov vortices explain phenomena in type {I\(\!\)I} superconductors \cite{Arovas:2019}.
The relativistic version, the abelian Higgs model, suggests production of the Nielsen-Olesen vortices \cite{Manton:2004tk} which can be a source of anisotropy of the early universe as cosmic strings \cite{Vilenkin:1994}.

Even in the presence of spatial inhomogeneity, such BPS structure is preserved.
In \((1+1)\)-dimensional inhomogeneous scalar field theories, various BPS solitons involving domain wall have been identified together with exact solutions and some of them have negative energy \cite{Adam:2019djg, Manton:2019xiq, Kwon:2021flc}.
In the \((1+2)\)-dimensional inhomogeneous abelian Higgs model, BPS structure has been established \cite{Tong:2013iqa}, and the BPS limit has been demonstrated to be extendable to the inhomogeneous abelian Chern-Simons Higgs model \cite{Han:2015tga}.

Recent careful studies of the BPS structures in the inhomogeneous abelian Chern-Simons Higgs model \cite{ImCSH} and the inhomogeneous abelian Higgs model \cite{ImAH} exhibit the nontrivial inhomogeneous BPS vacuum solution in addition to soliton solutions.
Although the newly found vacuum solution has nontrivial functional behaviors to accommodate the externally imposed inhomogeneity, global aspects of the configuration remain the same as those of the trivial constant vacuum of homogeneous field theories.
For instance, the vacuum solution has minimum zero energy, zero spin, neutral electric charge and zero magnetic flux.
Despite the importance of vacuum structures in field theories, this nontrivial vacuum has not been explicitly discussed with mathematical rigor in the literature.
On the other hand, it is also found that the BPS bound is tied to the sign of the inhomogeneous term added to the Lagrangian, so that multi-vortex solutions consist exclusively of either vortices or antivortices, ensuring that the energy is positive semidefinite.

Under the general finite source-energy condition
\(
	\sigma \in
	L^{2}(\mathbb{R}^{2})
\)
of the inhomogeneity, the proof of the existence and uniqueness of the BPS multi-vortex solutions encounters technical difficulties in the inhomogeneous abelian Higgs model as pointed out in Ref. \cite{Han:2015tga}.
This difficulty is physically reasonable considering the possibility of the absence of an inhomogeneous BPS vacuum or vortex solutions under arbitrary shape of impurity.
Thus, a mathematically rigorous proof can only be pursued under physically plausible assumptions about the inhomogeneity.
Once such a proof is achieved, the obtained results, including the assumed conditions will provide valuable theoretical and experimental insights for further physics studies.

In this work, we introduce a condition on the inhomogeneous function in terms of its upper bound of the \(L^{2}(\mathbb{R}^{2})\) norm.
This type of localized inhomogeneity is physically reasonable due to its consistency with the finiteness of the energy contribution of the impurity.
Under this assumed condition, we show the existence and uniqueness of both inhomogeneous vacuum and topological multi-vortex solutions of the inhomogeneous abelian Higgs model in the BPS limit.

The remaining part of the paper consists of four sections:
In section 2, we begin by briefly reviewing the inhomogeneous abelian Higgs model up to derivation of the Bogomolny equation rewritten in terms of dimensionless parameters.
Under a legitimate condition for inhomogeneity, we prove the existence and uniqueness of the newly found BPS vacuum solution in section 3.
In section 4, we extend the well-known existence and uniqueness for topological BPS multi-vortex solution \cite{Taubes:1979tm, Jaffe:1980mj} to the realm of inhomogeneity with the help of the same condition.
We conclude this paper in section 5 with some remarks on further work to be done, which are suggested by numerical works.

\section{Bogomolny Equation of Inhomogeneous Abelian Higgs Model}

This section introduces the model of our interest and summarizes briefly derivation of the Bogomolny equation with boundary condition.
In \((1+2)\)-dimensional spacetime with metric signature \((-,+,+)\), the inhomogeneous abelian Higgs model in the Bogomolny limit is described by the Lagrangian density \cite{Tong:2013iqa}
\begin{equation}
	\mathscr{L} =
	- \frac{1}{4} F_{\mu\nu} F^{\mu\nu}
	- \overline{ D_{\mu} \phi } D^{\mu} \phi
	- \frac{g^2}{2}
		( |\phi|^{2}
		- v^{2}(\boldsymbol{x}) )^{2}
	+ sg\sigma(\boldsymbol{x}) B
,\label{208}
\end{equation}
where \(\phi\) is a complex scalar field and \(A_\mu\) is a \(\mathrm{U}(1)\) gauge field whose gauge-covariant derivatives and field-strength are defined as
\begin{equation}
	D_{\mu} \phi =
	( \partial_{\mu} - i g A_{\mu} ) \phi
, \qquad
	\overline{D_{\mu} \phi} = 
	( \partial_{\mu} + i g A_{\mu} ) \bar{\phi}
, \qquad
	F_{\mu\nu} =
	\partial_{\mu} A_{\nu} - \partial_{\nu} A_{\mu}
.
\end{equation}
Note that square
\(
	v^{2}(\boldsymbol{x})
	= v_{0}^{2} + \sigma(\boldsymbol{x})
\)
of the vacuum expectation value consists of a constant piece \(v_{0}^{2}\) and a spatially inhomogeneous function \( \sigma (\boldsymbol{x}) \) which stands for a kind of impurities, e.g. impurities in condensed matter samples, and is assumed to vanish at spatial infinity
\begin{equation}
	\lim_{ |\boldsymbol{x}|\to\infty }
	\sigma (\boldsymbol{x}) = 0
.\label{206}
\end{equation}
Accordingly the last term including magnetic field \(B=F_{12}\) is also inhomogeneous, whose parameter \(s\) has either \(+1\) or \(-1\).

The energy is given by
\begin{equation}
	E
	=
	\int d^{2} x \, \bigg\{
		\frac{1}{2} |\boldsymbol{E}|^{2}
		+ \frac{1}{2} B^{2}
		+ |D_{0} \phi|^{2}
		+ |D_{i} \phi|^{2}
		+ \frac{g^2}{2} [
			|\phi|^{2}
			- v^{2} (\boldsymbol{x})
		]^{2}
		- sg \sigma(\boldsymbol{x}) B
	\bigg\}
,
\end{equation}
where \( \boldsymbol{E} \) is the electric field \( (\boldsymbol{E})^{i} = E_{i} = F_{i0} \). 
For static configurations in the Weyl gauge \(A^{0} = 0\), the energy can be reshuffled to
\begin{equation}
	E =
	\int d^{2} x \,
	\bigg\{
		| ( D_{1} + isD_{2} ) \phi|^{2}
		+ \frac{1}{2}
			\big|
				B + sg
				[ |\phi|^{2}
				- v^{2} (\boldsymbol{x}) ]
			\big|^{2}
		+ sg v_{0}^{2} B
	\bigg\}
,\label{204}
\end{equation}
where we have omitted a few boundary terms which do not contribute to the energy.
Then we get the BPS bound
\begin{equation}
	E \ge
	sg v_{0}^{2} \int d^{2} x \, B
	=
	sg v_{0}^{2} \Phi_\mathrm{B}
,\label{205}
\end{equation}
where \(\Phi_B\) is the magnetic flux.
The bound is saturated if the the first two terms in \eqref{204} vanish, i.e.,
\begin{equation}
	( D_{1} + isD_{2} ) \phi = 0
, \qquad
	B +
	sg [ |\phi|^{2}
		 - v^{2}(\boldsymbol{x}) ]
	= 0
,\label{207}
\end{equation}
which are called the Bogomolny equations.
Notice that, in obtaining the BPS bound, the form of the scalar potential in \eqref{208} is the same as in the homogeneous abelian Higgs model \cite{Bogomolny:1975de} except for the introduction of the inhomogeneity
\(
	v_{0}^{2}
	\to
	v^{2} (\boldsymbol{x}).
\)

In the homogeneous case with \( \sigma=0 \), the Lagrangian \eqref{208} is independent of \(s\) and hence the BPS bound can be obtained for both \(s=1\) and \(s=-1\) in a single theory.
Then we can get energy bound \(E \ge g v_{0}^{2} |\Phi_\mathrm{B}| \).
With \(\sigma \neq 0\), however, we have only one bound \eqref{205} for given \(s\), because it appears in the inhomogeneous term \( s\sigma(\boldsymbol{x})B \).
We set \(s=1\) from now on. 
Then the first Bogomolny equation in \eqref{207} implies that the zeros of the scalar field \(\phi\), if it has any, are isolated and only positive vortex number is possible, i.e., \( \Phi_\text{B} \ge 0 \) due to Poincar{\'e} lemma \cite{Jaffe:1980mj}.
Thus, the energy bound in \eqref{205} is positive semi-definite for solutions satisfying \eqref{207}~\cite{ImAH}.

The first Bogomolny equation in \eqref{207} is solved for spatial components of the \(\mathrm{U}(1)\) gauge field \(A_i\) and substitution of it into the second Bogomolny equation leads to a second order partial differential equation for scalar amplitude \( |\phi|(\boldsymbol{x}) \)
\begin{equation}
	\nabla^{2} \ln
	\frac{
		|\phi|^{2}
	}{
		\displaystyle \prod_{a=1}^{n}
		| \boldsymbol{x} - \boldsymbol{x}_{a} |^{2}
	}
	= 2g^{2} [ |\phi|^{2} - v^{2}(\boldsymbol{x}) ]
,\label{201}
\end{equation}
where
\(
	\{ \boldsymbol{x}_{a} \in \mathbb{R}^{2}
	\mid
	\phi( \boldsymbol{x}_{a} ) = 0
	\text{ and } 
	a = 1, 2, \cdots, n
	\}
\)
denote the zeros of \(\phi\) which are interpreted as \(n\) locations of vortices.
The right hand side of the equation with the help of the boundary behavior \eqref{206} dictates the boundary condition for finite energy
\begin{equation}
	\lim_{ |\boldsymbol{x} | \to \infty }
	|\phi|^{2}
	= v_{0}^{2}
.\label{202}
\end{equation}

It is convenient to rescale variables by introducing dimensionless quantities \( \tilde{\boldsymbol{x}} \), \(f\) and \(\varsigma\),
\begin{equation}
	\boldsymbol{x}
	= 
	\frac
		{ \tilde{\boldsymbol{x}} }
		{ \sqrt{2} g v_{0} }
, \quad
	|\phi|^{2} = v_{0}^{2} f
, \quad
	\sigma
	=
	v_{0}^{2} \varsigma
,
\end{equation}
that simplify the equation \eqref{201} and boundary condition \eqref{202} as
\begin{align}
	& \tilde{\nabla}^{2} \ln
	\frac{ f }
	{
		\displaystyle \prod_{a=1}^{n}
		| \tilde{ \boldsymbol{x} }
		- \tilde{ \boldsymbol{x} }_{a} |^{2}
	}
	=
	f  - \varsigma - 1 
,\label{203}
\\
	& \lim_{ |\tilde{\boldsymbol{x}}| \to \infty}
	f ( \tilde{\boldsymbol{x}} )
	= 1
.
\end{align}
The rescaled coordinates \( \tilde{\boldsymbol{x}} \) will be denoted without tilde symbol from now on for further simplicity.

\section{Inhomogeneous BPS Vacuum}

The inhomogeneous BPS vacuum corresponds to the solution with zero vorticity \(n=0\), and the equation \eqref{203} reduces to
\begin{align}
	& \nabla^{2} u
	=
	e^{u} - \varsigma - 1
,\label{301}
\\
	& \lim_{ |\boldsymbol{x}| \to \infty }
	u(\boldsymbol{x})
	= 0
,\label{308}
\end{align}
where \( u = \ln f \) is well-defined on entire plane \(\mathbb{R}^{2}\).

For the homogeneous case \(\varsigma=0\), the solution of \eqref{301} is trivial \(u=0\) and this constant Higgs vacuum solution is the minimum energy configuration with vanishing energy.
On the other hand, once a nontrivial function \( \varsigma(\boldsymbol{x}) \) is turned on, the constant configuration \(u=0\) no longer solves the equation \eqref{301}, and thus we should find a nontrivial vacuum solution subject to the boundary condition \eqref{308}. Thus it is physically important to establish the existence and the uniqueness of the solution of the BPS vacuum equation \eqref{301}.

The current section addresses this question.
We show the existence and the uniqueness of the vacuum solution obeying the equation \eqref{301} under the boundary condition \eqref{308} provided that the inhomogeneity \( \varsigma(\boldsymbol{x}) \) is localized and not too strong.
The local distribution of the externally given inhomogeneity is also a physically acceptable assumption and mathematically compatible with some finite upper bound of \(L^{2}(\mathbb{R}^{2})\) norm.
To be specific, we prove the following theorem.
\begin{theorem}[Existence and uniqueness of inhomogeneous BPS vacuum solution]
\label{thm:vac}
	Assume that the inhomogeneity \( \varsigma \in L^{2}(\mathbb{R}^{2}) \) satisfies the inequality
	\begin{equation}
		\lVert \varsigma \rVert_{2}
		=
		\bigg[
			\int d^{2}x \,
			| \varsigma |^{2}
		\bigg]^{ \frac{1}{2} }
		<
		\sqrt{\frac{2}{\pi}}
	,\label{302}
	\end{equation}
	then there exists a unique weak solution \( u \in W^{1,2} (\mathbb{R}^{2}) \) for the equation \eqref{301}, where \( W^{1,2} (\mathbb{R}^{2}) \) denotes the Sobolev space on \(\mathbb{R}^2\) whose norm is defined as
	\begin{equation}
		\lVert u \rVert_{1,2}
		= \bigg(
			\int d^{2} x \, |u|^{2} +
			\int d^{2} x \, |\partial_{i} u|^{2}
		\bigg)^{\frac{1}{2}}
	.
	\end{equation}
\end{theorem}
The equation \eqref{301} is the Euler-Lagrange equation of a real-valued functional
\begin{equation}
	S[u] = 
	\int d^{2} x \,
	\bigg[
		\frac{1}{2} (\partial_{i} u)^{2}
		+ e^{u} - \varsigma u - u - 1
	\bigg]
,\label{303}
\end{equation}
and its critical point corresponds to the weak solution of the equation.
As in the homogeneous case \cite{Taubes:1979tm}, the functional \eqref{303} is well-defined on entire \( W^{1,2} (\mathbb{R}^{2}) \).
\begin{proposition}
	The functional \(S[u]\) is finite for all \( u \in W^{1,2} (\mathbb{R}^{2}) \).
\end{proposition}
\begin{proof}
	The absolute value of the functional \eqref{303} is bounded by the triangle inequality
	\begin{equation}
		| S[u] |
		\le
		\frac{1}{2}
		\int d^{2} x \, |\partial_{i} u|^{2}
		+
		\int d^{2} x \, | e^{u} - u - 1 |
		+
		\int d^{2} x \, |\varsigma u|
	.\label{304}
	\end{equation}
	Since the first two terms are identical to the homogeneous case, they are finite.
	Application of H{\"o}lder's inequality to the last integral of \eqref{304} gives
	\begin{equation}
		\int d^{2} x \, |\varsigma u|
		\le
		\lVert \varsigma \rVert_{2}
		\lVert u \rVert_{2}
		\le
		\lVert \varsigma \rVert_{2}
		\lVert u\rVert_{1,2}
	\end{equation}
	which is finite for all \( u \in W^{1,2} (\mathbb{R}^{2}) \) by the assumed finiteness of inhomogeneity \eqref{302}.
\end{proof}
In general, existence and uniqueness of the critical point are studied by virtue of the extreme value theorem and the convexity.
As a sufficient condition, we use the following proposition.
\begin{proposition}
\label{thm:J-T}
	Let \(E\) be a real reflexive Banach space, \( S: E \to \mathbb{R} \) be a real functional.
	If \(S\) is strictly convex, Gateaux differentiable and satisfies 
	\begin{equation}
		\inf_{ \lVert u \rVert = R }
		( \nabla _{u} S ) [u] \ge \delta
		\qquad \text{for some }
		R, \delta > 0
	,\label{305}
	\end{equation}
	the global minimizer of \(S\) uniquely exists in the ball
	\(
		\{
			u\in E	\mid
			\lVert u \rVert
			< R
		\}
	\).
\end{proposition}
\begin{proof}
	This is a direct consequence from Proposition VI.7.8 to Proposition VI.8.6 of \cite{Jaffe:1980mj}, pp. 275--279.
\end{proof}
Since \( W^{1,2} (\mathbb{R}^{2}) \) is a Hilbert space which is naturally reflexive, Proposition \ref{thm:J-T} guarantees the existence and uniqueness of the global minimizer of the functional \eqref{303}.
The rest of this section shows that each of the conditions of Proposition \ref{thm:J-T} is indeed valid for the functional \eqref{303}, which would complete the proof of Theorem \ref{thm:vac}.
\begin{proposition}
	The functional \(S[u]\) is strictly convex on \( W^{1,2} (\mathbb{R}^{2}) \).
\end{proposition}
\begin{proof}
	The functional \eqref{303} is rearranged as
	\begin{equation}
		S[u] = 
		\frac{1}{2}
		\int d^{2} x \,
			(\partial_{i} u)^{2}
		+
		\int d^{2} x \,
			( e^{u} - u - 1 )
		-
		\int d^{2} x \,
			(\varsigma u)
	.
	\end{equation}
	The first integral is strictly convex due to the strict convexity of the quadratic function and the linearity of the partial derivative.
	The second integral is also strictly convex since the real function \( x \mapsto (e^{x} - x - 1) \) is strictly convex for all \(x\in\mathbb{R}\).
	The linearity of integration implies that the last integral is linear in \(u\).	
	To sum up, the functional \eqref{303} is strictly convex.
\end{proof}
\begin{proposition}
	The functional \(S[u]\) is Gateaux differentiable on \( W^{1,2} (\mathbb{R}^{2}) \).
\end{proposition}
\begin{proof}
	The definition of Gateaux derivative reads
	\begin{equation}
	\begin{aligned}[b]
		( \nabla _{h} S ) [u]	
		& =
		\lim_{\epsilon \to 0}
		\frac
			{ S[u+\epsilon h] - S[u] }
			{\epsilon}
	\\	& =
		\int d^{2} x \,
		\Big[
			( \partial_{i} u ) ( \partial_{i} h )
			+ h( e^{u} -1 )
		\Big]
		+ \lim_{\epsilon \to 0}
		\int d^{2} x \,
		\bigg[
			\frac{1}{2} \epsilon ( \partial_{i} h )^{2}
			+ \frac
				{ e^{\epsilon h} - \epsilon h - 1 }
				{\epsilon}
		 e^{u}
		\bigg]
	\\	& \quad
		- \int d^{2} x \,
		( \varsigma h )
	.
	\end{aligned}
	\end{equation}
	Two integrals in the second line are identical to the homogeneous case \( \varsigma(\boldsymbol{x}) = 0 \); the first term is finite and the second term vanishes for any \(h, u \in W^{1,2} (\mathbb{R}^{2})\).
	Application of H{\"o}lder's inequality with the assumption \eqref{302} shows that the last line is finite, hence the Gateaux derivative
	\begin{equation}
		( \nabla _{h} S ) [u]
		=
		\int d^{2} x \,
		\Big[
			( \partial_{i} u ) ( \partial_{i} h )
			+ h( e^{u} -1 )
			- \varsigma h
		\Big]
	\label{306}
	\end{equation}
	exists for all \(h, u \in W^{1,2} (\mathbb{R}^{2})\).
\end{proof}

For the last condition \eqref{305}, observe the lower bound of the radial derivative from \eqref{306}
\begin{equation}
\begin{aligned}[b]
	( \nabla _{u} S ) [u]
	& =
	\int d^{2} x \,
	\Big[
		( \partial_{i} u )^{2}
		+ u ( e^{u} - 1 )
		- \varsigma u
	\Big]
\\	& \ge
	\int d^{2} x \,
	| \partial_{i} u |^{2}
	+
	\int d^{2} x \, 
	u (e^{u} - 1)
	-
	\int d^{2} x \, 
	|\varsigma| |u|
\\	& \ge
	\lVert u \rVert^{2}_{1, 2}
	- \lVert u \rVert^{2}_{2}
	+
	\int d^{2} x \, 
	u (e^{u} - 1)
	-
	\lVert \varsigma \rVert_{2}
	\lVert u \rVert_{2}
,\label{307}
\end{aligned}
\end{equation}
where the triangle inequality and H{\"o}lder's inequality are applied.
Unlike the homogeneous cases, it is necessary to find a specific value of lower bound of the third term in the last line of \eqref{307}.
\begin{lemma}
\label{lem:exp}
	The following inequality holds for any \( u \in W^{1,2} (\mathbb{R}^{2}) \)
	\begin{equation}
		\int d^{2} x \, 
		u (e^{u} - 1)
		\ge
		\frac
		{
			\lVert u \rVert^{4}_{2}
		}
		{
			\lVert u \rVert^{2}_{2}
			+
			\sqrt{\dfrac{\pi}{2}}
			\lVert u \rVert^{3}_{1,2}
		}
	.
	\end{equation}
\end{lemma}
\begin{proof}
	On the region where \( u < 0 \), the integrand is bounded below by
	\begin{equation}
		u (e^{u} - 1)
		= |u| ( 1 - e^{ -|u| } )
		\ge |u|
		\bigg(
			1 - \frac{1}{ 1 + |u| }
		\bigg)
		= \frac{ |u|^{2} }{ 1+|u| }
	,
	\end{equation}
	and the bound is also valid on the region where \( u \ge 0 \) as
	\begin{equation}
		u (e^{u} - 1)
		= |u| ( e^{|u|} - 1 )
		\ge |u|^{2}
		\ge \frac{ |u|^{2} }{ 1+|u| }
	.
	\end{equation}
	The integral of the lower bound is estimated by rearranging H{\"o}lder's inequality
	\begin{equation}
		\bigg(
			\int d^{2} x \, 
			\frac{ |u|^{2} }{ 1+|u| }
		\bigg)^{1/2}
		\ge
		\frac{\displaystyle
			\int d^{2} x \, |u|^{2}
		}{\displaystyle
			\bigg[
				\int d^{2} x \, 
				\Big(
					|u|^{2} + |u|^{3}
				\Big)
			\bigg]^{1/2}
		}
		=
		\frac{
			\lVert u \rVert_{2}^{2}
		}{
			\big(
				\lVert u \rVert_{2}^{2}
				+
				\lVert u \rVert_{3}^{3}
			\big)^{1/2}
		}
	.
	\end{equation}
	The \(L^{3}(\mathbb{R}^{2})\) norm in the denominator is handled by the Hausdorff-Young inequality for \(p=3\) together with H{\"o}lder's inequality for \( (p',~q') = (4/3,~4)\)
	\begin{equation}
	\begin{aligned}[b]
		\lVert u \rVert_{3}
		\le
		\lVert \hat{u} \rVert_{3/2}
		& \le
		\bigg[
			\Big\lVert
				| \hat{u} |^{3/2}
				( 1 + k^{2} )^{3/4}
			\Big\rVert_{4/3}
			\times
			\Big\lVert
				( 1 + k^{2} )^{-3/4}
			\Big\rVert_{4}
		\bigg]^{2/3}
	\\	& =
		\bigg[
			\int d^{2} k \,
			| \hat{u} |^{2}
			( 1 + k^{2} )
		\bigg]^{1/2}
		\times
		\bigg[
			\int d^{2} k \,
			( 1 + k^{2} )^{-3}
		\bigg]^{1/6}
		=
		\sqrt[6]{ \frac{\pi}{2} }
		\lVert u \rVert_{1,2}
	,
	\end{aligned}
	\end{equation}
	where \(\hat{u}\) is the Fourier transform of \(u\) and the unitarity of Fourier transform in \(L^{2}(\mathbb{R}^{2})\) is used,
	\begin{equation}
		\int d^{2} k \,
		| \hat{u} |^{2}
		( 1 + k^{2} )
		=
		\lVert \hat{u} \rVert_{2}^{2}
		+
		\lVert k^{i}\hat{u} \rVert_{2}^{2}
		=
		\lVert u \rVert_{2}^{2}
		+
		\lVert \partial_{i} u \rVert_{2}^{2}
		=
		\lVert u \rVert_{1,2}^{2}
	.
	\end{equation}
	Collection of the aforementioned inequalities concludes that
	\begin{equation}
		\int d^{2} x \, 
		u (e^{u} - 1)
		\ge
		\int d^{2} x \, 
		\frac{ |u|^{2} }{ 1+|u| }
		\ge
		\frac{
			\lVert u \rVert_{2}^{4}
		}{
			\lVert u \rVert_{2}^{2}
			+
			\lVert u \rVert_{3}^{3}
		}
		\ge
		\frac
		{
			\lVert u \rVert^{4}_{2}
		}
		{
			\lVert u \rVert^{2}_{2}
			+
			\sqrt{\dfrac{\pi}{2}}
			\lVert u \rVert^{3}_{1,2}
		}
	.
	\end{equation}
\end{proof}
\begin{proposition} \label{prop37}
	There exist \(R, \delta > 0\) satisfying the inequality \eqref{305}.
\end{proposition}
\begin{proof}
	Since
	\(
		0
		\le \lVert u \rVert_{2}
		\le \lVert u \rVert_{1, 2}
	\),
	there exists \( t \in [0, 1] \) satisfying \( \lVert u \rVert_{2} = t \lVert u \rVert_{1,2} \).
	Denote \( r = \lVert u \rVert_{1,2} \) and rewrite the right hand side of the inequality \eqref{307} using Lemma \ref{lem:exp}
	\begin{equation} \label{3.20}
	\begin{aligned}[b]
		( \nabla _{u} S ) [u]
		& \ge
		( 1 - t ^{2} ) r^{2}
		+
		\frac{
			t^{4} r^{2}
		}{
			t^{2}
			+
			\sqrt{ \dfrac{\pi}{2} } r
		}
		-
		t \, r \lVert \varsigma \rVert_{2}
	\\	& =
		( 1 - t^{2} ) r^{2}
		+
		\bigg[
			\sqrt{\frac{2}{\pi}}
			t^{4}
			-
			t \lVert \varsigma \rVert_{2}
		\bigg] r
		-
		\frac{ 2 t^{6} }{\pi}
		+
		\frac{ 2 t^{8} }{\pi}
		\frac{ 1 }{
			t^{2}
			+
			\sqrt{ \dfrac{\pi}{2} } r
		}
	\\	& \ge
		( 1 - t^{2} ) r^{2}
		+
		\bigg[
			\sqrt{\frac{2}{\pi}}
			t^{4}
			-
			t \lVert \varsigma \rVert_{2}
		\bigg] r
		-
		\frac{ 2 t^{6} }{\pi}
	.
	\end{aligned}
	\end{equation}
	The right hand side is a monotonically decreasing polynomial in \(t\),
% with respect to \( t \in [0, 1] \), 
	and thus is minimized at \(t=1\).
	Then,
	\begin{equation}
		( \nabla _{u} S ) [u]
		\ge
		\bigg[
			\sqrt{\frac{2}{\pi}}
			-
			\lVert \varsigma \rVert_{2}
		\bigg] r
		-
		\frac{2}{\pi}
	.
	\end{equation}
	Since the linear coefficient of the lower bound is positive definite by the assumption \eqref{302}, the inequality
	\begin{equation}
		\inf_{ \lVert u \rVert_{1,2} = R }
		( \nabla _{u} S ) [u] \ge
		\bigg[
			\sqrt{\frac{2}{\pi}}
			-
			\lVert \varsigma \rVert_{2}
		\bigg] R
		-
		\frac{2}{\pi}
		\ge \delta > 0
	\end{equation}
	holds for a sufficiently large \(R\).
\end{proof}

As a concluding remark, observe the right hand side of \eqref{301} is square-integrable
\begin{equation}
	\nabla^{2} u
	=
	e^{u} - \varsigma - 1
	=
	(e^{u} - u - 1) + u - \varsigma
	\in L^{2}(\mathbb{R}^{2})
.
\end{equation}
Integration by parts twice gives
\begin{equation}
	\sum_{i,j}
	\int d^{2}x \,
	| \partial_{i} \partial_{j} u |^{2}
	=
	\sum_{i,j}
	\int d^{2}x \,
	( \partial_{i}^{2} u )
	( \partial_{j}^{2} u )
	=
	\int d^{2}x \,
	| \nabla^{2} u |^{2}
	< \infty
\end{equation}
which implies \(u \in W^{2,2} (\mathbb{R}^{2}) \), and the Morrey-Sobolev embedding \cite{Evans:2010} shows that the obtained weak solution is indeed continuous \(u \in C^{0}(\mathbb{R}^{2})\).
Further analysis on the regularity of the weak solution can be addressed only after smoothness of inhomogeneity \(\varsigma(\boldsymbol{x})\) is provided by physics.

\section{BPS Multi-Vortex Solutions}

In this section, we prove the existence and the uniqueness of multi-vortex solutions of \eqref{203} with nonzero vorticity \(n\).
As in the homogeneous case \cite{Taubes:1979tm}, let us define background functions
\begin{equation} \label{400}
	u_{0} =
	- \ln \prod_{a=1}^{n}
		\bigg(
			1 + \frac
				{ \lambda }
				{ | \boldsymbol{x}
					- \boldsymbol{x}_{a} |^{2}
				}
		\bigg)
, \qquad
	g_{0} =
	\sum_{a=1}^{n}
	\frac{ 4\lambda }
		{ ( | \boldsymbol{x}
			- \boldsymbol{x}_{a} |^{2}
			+ \lambda )^{2}
		}
,
\end{equation}
where \( \lambda > 4n \) is a free parameter. Then, \eqref{203} becomes
\begin{equation}
	\nabla^{2} u
	= e^{u} e^{u_0} + g_{0} - \varsigma - 1
, \qquad
	\lim_{ |\boldsymbol{x}| \to \infty }
	u(\boldsymbol{x})
	= 1
,\label{401}
\end{equation}
where \( u = \ln f - u_{0} \).
All the singular behaviors are traded from \(\ln f\) to the background \(u_{0}\) so that subtraction of \(u_0\) from \(\ln f\) lets \(u\) be nonsingular.
For the homogeneous case with \( \varsigma(\boldsymbol{x}) = 0 \), the existence and the uniqueness of the vortex solution for \eqref{401} are proven both on \(\mathbb{R}^2\) \cite{Taubes:1979tm} and on a lattice with periodic boundary condition \cite{Wang:1992}.
In \cite{Han:2015tga}, the existence and the uniqueness for the inhomogeneous case are proved on a lattice domain with periodic boundary condition, under the finite source-energy condition \( \varsigma \in L^{2}(\mathbb{R}^{2}) \). This section presents the existence theorem for the equation \eqref{401} on \(\mathbb{R}^2\) if \( \varsigma \) satisfies the condition \eqref{302}.
\begin{theorem}[Existence and uniqueness of BPS soliton]
\label{thm:sol}
	Assume that the inhomogeneity \( \varsigma \in L^{2}(\mathbb{R}^{2}) \) satisfies the inequality
	\begin{equation}
		\lVert \varsigma \rVert_{2}
		<
		\sqrt{\frac{2}{\pi}}
	,\label{402}
	\end{equation}
	then there exists a unique weak solution \( u \in W^{1,2} (\mathbb{R}^{2}) \) for the equation \eqref{401}.
\end{theorem}
As in the previous section, we prove the theorem by defining a real-valued action which produces \eqref{401} as its Euler-Lagrange equation,
\begin{equation}
	S[u] = 
	\int d^{2} x \,
	\bigg[
		\frac{1}{2} (\partial_{i} u)^{2}
		+ (e^{u} - 1) e^{u_0}
		- (1 - g_{0}) u
		- \varsigma u
	\bigg]
.\label{403}
\end{equation}
Now we are going to show that the conditions for Proposition \ref{thm:J-T} holds as in the vacuum case \(n=0\) with some modification, which will prove the Theorem \ref{thm:sol}.

\begin{proposition}
	The functional \(S[u]\) is finite for all \( u \in W^{1,2} (\mathbb{R}^{2}) \).
\end{proposition}
\begin{proof}
	The absolute value of the functional \eqref{403} is bounded by the triangle inequality
	\begin{equation}
	\begin{aligned}[b]
		| S[u] |
		& \le
		\frac{1}{2}
		\int d^{2} x \, |\partial_{i} u|^{2}
		+
		\int d^{2} x \, | (e^{u} - u - 1) e^{u_0} |
		+
		\int d^{2} x \, | (e^{u_0} + g_{0} - 1) u |
	\\	& \quad +
		\int d^{2} x \, | \varsigma u|
	,
	\end{aligned}
	\end{equation}
	where the first line is identical to the homogeneous case and hence finite.
	Application of H{\"o}lder's inequality with the assumption \eqref{402} to the second line completes the proof.
\end{proof}
\begin{proposition}
	The functional \(S[u]\) is strictly convex on \( W^{1,2} (\mathbb{R}^{2}) \).
\end{proposition}
\begin{proof}
	The functional \eqref{403} is rearranged as
	\begin{equation}
		S[u] = 
		\frac{1}{2}
		\int d^{2} x \,
			(\partial_{i} u)^{2}
		+
		\int d^{2} x \,
			( e^{u} - u - 1 ) e^{u_0}
		-
		\int d^{2} x \,
			( 1 - g_{0} - e^{u_0} + \varsigma ) u
	.
	\end{equation}
	By the same logic as in the vacuum case, the first integral is strictly convex.
	As the real function \( x \mapsto (e^{x} - x - 1)e^{x_0} \) is strictly convex everywhere for arbitrary constant \(x_{0}\in\mathbb{R}\), so is the second integral.
	Since the last integral is linear in \(u\), the functional \eqref{403} is strictly convex.
\end{proof}
\begin{proposition}
	The functional \(S[u]\) is Gateaux differentiable on \( W^{1,2} (\mathbb{R}^{2}) \).
\end{proposition}
\begin{proof}
	The definition of Gateaux derivative reads
	\begin{equation}
	\begin{aligned}[b]
		( \nabla _{h} S ) [u]
		& =
		\int d^{2} x \,
		\Big[
			( \partial_{i} u ) ( \partial_{i} h )
			+ h( e^{u} -1 ) e^{u_0}
			- h( 1 - g_{0} - e^{u_0} )
		\Big]
	\\	& \quad
		+ \lim_{\epsilon \to 0}
		\int d^{2} x \,
		\bigg[
			\frac{1}{2} \epsilon ( \partial_{i} h )^{2}
			+ \frac
				{ e^{\epsilon h} - \epsilon h - 1 }
				{\epsilon}
		 e^{u} e^{u_0}
		\bigg]
		- \int d^{2} x \,
		(\varsigma h)
	,
	\end{aligned}
	\end{equation}
	and the obtained derivative is finite by the direct application of the homogeneous case with the help of the assumption \eqref{402}.
	Thus the Gateaux derivative
	\begin{equation}
		( \nabla _{h} S ) [u]
		=
		\int d^{2} x \,
		\Big[
			( \partial_{i} u ) ( \partial_{i} h )
			+ h( e^{u} -1 ) e^{u_0}
			- h( 1 - g_{0} - e^{u_0} )
			- \varsigma h
		\Big]
	\end{equation}
	exists for all \(h, u \in W^{1,2} (\mathbb{R}^{2})\).
\end{proof}
\begin{lemma} \label{lemma2}
	Let \(u_0\) and \(g_0\) be defined by \eqref{400} with \(\lambda > 4n\).
	Then for all \( u \in W^{1,2}(\mathbb{R}^2) \),
	\begin{equation}
	u( e^{u + u_0} + g_0 - 1 ) 
	\ge
	\left(
		1 - \frac{4n}\lambda
	\right)
	\frac{|u|^2}{1 + |u|^2}
	- \frac\lambda{16n} (u_0 + g_0)^{2}
	.
	\end{equation}
\end{lemma}
\begin{proof}
	If \( u(x) \ge 0 \), from (5.13) of \cite{Taubes:1979tm}, we have
	\begin{equation} \label{beta}
		u( e^{u + u_{0}} + g_{0} - 1 ) 
		\ge
		\beta \frac{|u|^2}{1 + |u|^2} - \frac{1}{4(1-\beta)} (u_{0} + g_{0})^{2}
	,
	\end{equation}
	where \(\beta\) is an arbitrary constant with \( \beta \in (0,1) \). 
	If \( u(x)\le 0 \), from (5.4) and (5.16) of \cite{Taubes:1979tm}, we have
	\begin{equation}
		u( e^{u + u_0} + g_0 - 1 ) 
		\ge
		\frac{\lambda - 4n}{\lambda}
		\frac{|u|^2}{1 + |u|^2}
	.
	\end{equation}
	Since \( \lambda > 4n\), we can choose \( \beta = 1 - 4n/\lambda \) in \eqref{beta} and hence for any \( u \in W^{1,2}(\mathbb{R}^2) \),
	\begin{equation}
		u( e^{u + u_0} + g_0 - 1 ) 
		\ge
		\frac{\lambda - 4n}{\lambda}
		\frac{|u|^2}{1 + |u|^2}
		- \frac\lambda{16n} (u_0 + g_0)^2
	,
	\end{equation}
	which completes the proof.
\end{proof}
\begin{proposition}
	There exist \(R, \delta > 0\) satisfying the inequality \eqref{305}.
\end{proposition}
\begin{proof}
From Lemma \ref{lemma2}, we have
	\begin{equation}
	\begin{aligned}[b]
		( \nabla _{u} S ) [u]
		& =
		\int d^{2} x \,
		\Big[
			( \partial_{i} u )^{2}
			+ u( e^{u + u_{0}} + g_{0} - 1 )
			- u \varsigma
		\Big]
	\\	& \ge
		\int d^{2} x \,
		( \partial_{i} u )^{2}
		+
		\frac{ \lambda - 4n }{ \lambda }
		\int d^{2} x \, 
		\frac{ |u|^{2} }{ 1+|u| }
		-
                \frac\lambda{16n}
		\int d^{2} x \,
		( u_{0} + g_{0} )^{2}
		-
		\int d^{2} x \, 
		|u| |\varsigma|
	.
	\end{aligned}
	\end{equation}
	Denote \( \lVert u \rVert_{2} = t \lVert u \rVert_{1,2} = t\,r \).
    Applying the Lemma \ref{lem:exp} to the second integral and repeating the same procedure given in \eqref{3.20} of Proposition \ref{prop37}, we get
	\begin{equation}
	\begin{aligned}[b]
		( \nabla _{u} S ) [u]
		& \ge
		( 1 - t^{2} )r^{2}
		+
		\frac{ \lambda - 4n }{ \lambda }
		\frac{
			t^{4} r^{2}
		}{
			t^{2}
			+
			\sqrt{ \dfrac{\pi}{2} } r
		}
		-
                \frac\lambda{16n}
		\int d^{2} x \,
		( u_{0} + g_{0} )^{2}
		-
		t r
		\lVert \varsigma \rVert_{2}
%	\\	& \ge
%		( 1 - t^{2} ) r^{2}
%		+
%		\bigg[
%			\sqrt{\frac{2}{\pi}}
%			\frac
%				{ \lambda - 4n }
%				{ \lambda }
%			t^{4}
%			-
%			t \lVert \varsigma \rVert_{2}
%		\bigg] r
%		- \bigg[
%			\frac2\pi
%			\frac
%				{ \lambda - 4n }
%				{ \lambda }
%			t^{6}
%			+
%			\frac\lambda{16n}
%			\int d^{2} x \,
%			( u_{0} + g_{0} )^{2}
%		\bigg]
	\\	& \ge
		\bigg[
			\sqrt{\frac{2}{\pi}}
			\frac
				{ \lambda - 4n }
				{ \lambda }
			-
			\lVert \varsigma \rVert_{2}
		\bigg] r
		- \bigg[
			\frac2\pi
			\frac
				{ \lambda - 4n }
				{ \lambda }
			+
			\frac\lambda{16n}
			\int d^{2} x \,
			( u_{0} + g_{0} )^{2}
		\bigg]
	,\label{404}
	\end{aligned}
	\end{equation}
	Given an inhomogeneity function \( \varsigma \) satisfying \eqref{402}, we can choose sufficiently large \( \lambda > 4n \) so that the coefficient of \(r\) becomes positive.
	Since the second term is a constant independent of \(r\), the inequality
%	where the second inequality is obtained by factorization of the second term while the positive-definite remainders are removed, and the third inequality follows from the monotonicity with respect to \(t \in [0,1] \).
%	The right hand side of \eqref{404} is linear to \(r\) and the coefficient is positive by taking a sufficient large parameter \(\lambda > 4n\) with the help of \eqref{402}, where the constant term is a fixed number.
%	Hence the inequality
	%
	\begin{equation}
		\inf_{ \lVert u \rVert_{1,2} = R }
		( \nabla _{u} S ) [u] \ge
		\bigg[
			\sqrt{\frac{2}{\pi}}
			\frac
				{ \lambda - 4n }
				{ \lambda }
			-
			\lVert \varsigma \rVert_{2}
		\bigg] R
		- \bigg[
			\frac2\pi
			\frac
				{ \lambda - 4n }
				{ \lambda }
			+
			\frac\lambda{16n}
			\int d^{2} x \,
			( u_{0} + g_{0} )^{2}
		\bigg]
		\ge \delta > 0
	\end{equation}
	holds for a sufficiently large \(R\).
\end{proof}

As in the vacuum case, we remark that the obtained weak solution is indeed continuous \(u \in C^{0}(\mathbb{R}^{2})\), while smoothness of inhomogeneity \(\sigma(\boldsymbol{x})\) of a given physical system needs to be specified for the analysis on the regularity of the weak solution.

\section{Conclusions and Discussion}

In this work, we have considered the Bogomolny equation \eqref{201} of the inhomogeneous abelian Higgs model \eqref{208} with critical quartic coupling.
Under the condition \eqref{302} on the inhomogeneity, which becomes
\(
	\int d^{2}x \,
	| \sigma (\boldsymbol{x}) |^{2}
	<
	v_{0}^{2} / \pi g^{2}
\)
in the original variables,
%\eqref{302} of the inhomogeneity, 
the existence and the uniqueness of the newly found nontrivial inhomogeneous BPS vacuum solution with vanishing energy \cite{ImAH} was proved by using variational approach.
% in which the condition \eqref{301} is sufficient.
We also proved the existence and the uniqueness of BPS multi-vortex solutions on \(\mathbb{R}^2\) in the presence of inhomogeneity.

Although we assumed the condition \eqref{302} in this paper, it should not be essential for the existence of the solutions.
In fact, in \cite{ImAH}, we considered Gaussian inhomogeneity,
\begin{equation}
	\varsigma(r)
	=
	- \beta
	e^{ - \alpha^{2} r^{2} },
\end{equation}
and numerically found the vacuum solution and rotationally symmetric vortex solutions for various values of parameters \(\alpha\) and \(\beta\). 
In particular, we were able to obtain numerical solutions even in the case that the condition \eqref{302} does not hold, as seen in Figure \ref{fig:501} where \( |\beta| \) should be less than \( \sqrt{2} \alpha/\pi\approx 0.1 \) for \eqref{302} to be satisfied.
%
%This proof, however, is not the end of story but a beginning for complete what-to-do mathematically.
%In \cite{ImAH}, numerical works for rotationally symmetric configurations in the presence of the Gaussian inhomogeneity
%%
%\begin{equation}
%	\varsigma(r)
%	=
%	- \beta
%	e^{ - \alpha^{2} r^{2} }
%\end{equation}
%%
%are performed with various values of parameters \(\alpha\) and \(\beta\).
%Numerous rotationally symmetric numerical vacuum and vortex solutions are found, which are approximate due to possible numerical error in the viewpoint of exact solutions but seem to be valid even as the solutions of arbitrary shape without rotational symmetry.
%A noteworthy point is that numerical solutions exist beyond the assumed conditions \eqref{302} and \eqref{402} as subsets, e.g. \(\alpha^{2} = 0.05\) and 
%%
%\(
%	|\beta|
%	<
%	1 / \sqrt{10} \pi
%\),
%%
%as illustrated in the Figure \ref{fig:501}.
%
\begin{figure}
    \centering
    \includegraphics[
        page=1,
        width=0.75\textwidth
    ]{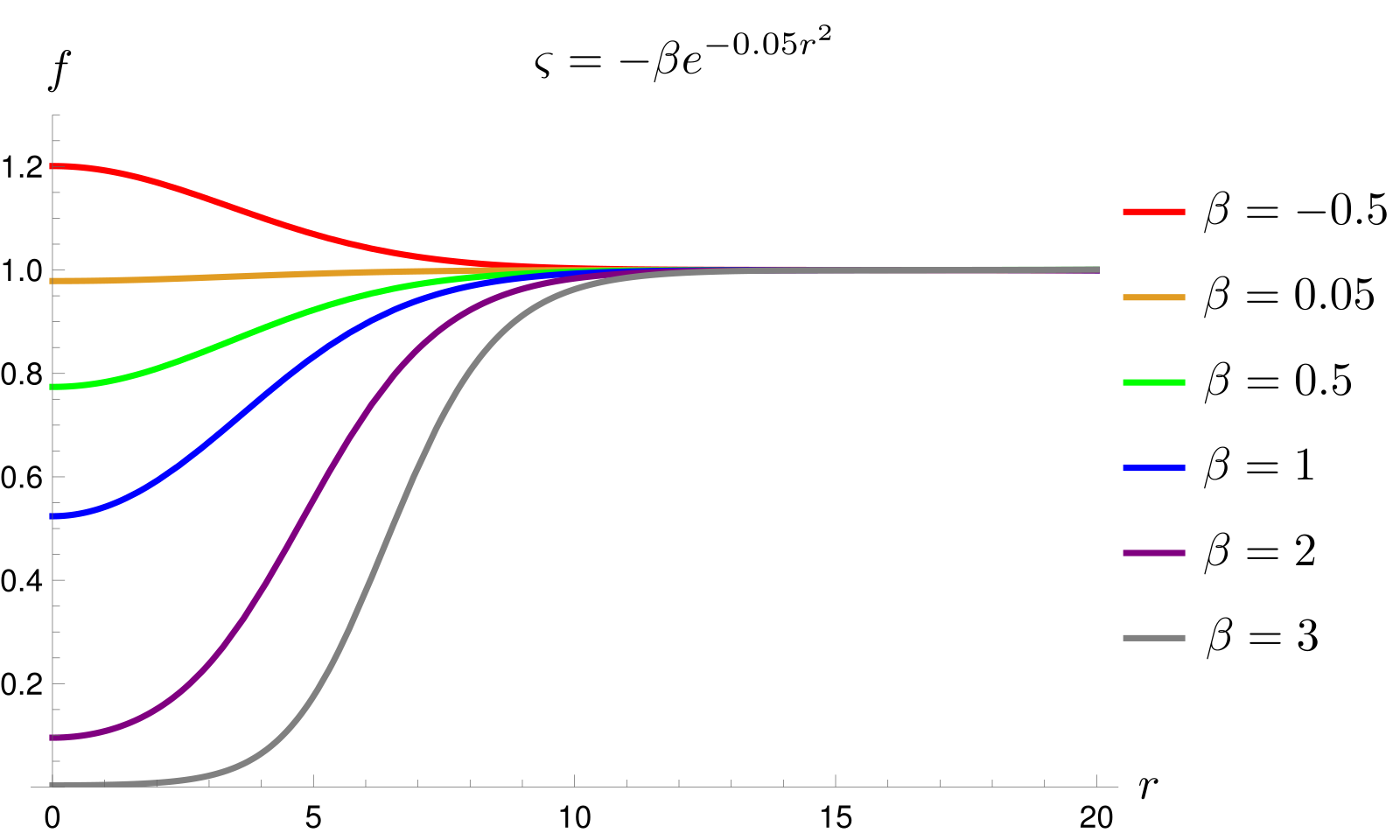}
    \caption{Inhomogeneous BPS vacuum solutions obtained by numerical works for various values of a parameter \(\beta\).}
	\label{fig:501}
%	\vspace{-3em}
\end{figure}
This strongly suggests the existence of solutions for larger inhomogeneities which violate \eqref{302}.
It would be interesting if one can prove the existence under a weaker condition such as the finite norm condition \( \lVert \sigma \rVert_{2} < \infty \).
Alternatively, there could exist some upper bound in \( \lVert \sigma \rVert_{2} \) beyond which solution no longer exists.
In the latter case, what would be the physical as well as mathematical significance of such value?
In order to answer these questions, one might have to employ other approaches such as iterative methods \cite{Wang:1992, Spruck:1995yy, Caffarelli:1995dz, Chae:2000tv}.

%The numerical works strongly suggest further rigorous analytic studies on existence and uniqueness of inhomogeneous BPS vacuum and multi-vortex solutions for some weaker conditions involving the imposed conditions \eqref{302} and \eqref{402} as subsets.
%As pointed in \cite{Han:2015tga}, the general finite source-energy condition
%%
%\(
%	\int d^{2}x \,
%	| \sigma (\boldsymbol{x}) |^{2}
%	<
%	\infty
%\)
%%
%does not provide an existence and uniqueness theorem till now.
%Therefore, it is intriguing to find an improved borderline which guarantees the existence and uniqueness of these numerically claimed BPS configurations and the obtained results.
%For this purpose, further works probably employ different approaches, e.g. iterative method \cite{Wang:1992, Spruck:1995yy, Caffarelli:1995dz, Chae:2000tv}, which may be of great importance.

BPS multi-soliton solutions are found in other field theories in two spatial dimensions \cite{Manton:2004tk} such as the abelian Chern-Simons Higgs model \cite{Hong:1990yh, Jackiw:1990aw} which supports nontopological solitons and nontopological vortices in addition to topological vortices \cite{Jackiw:1990pr}.
Mathematically rigorous studies have established existence of BPS Chern-Simons soliton solutions \cite{Wang:1991na, Spruck:1995yy, Caffarelli:1995dz, Chae:2000tv, Kim:2000hx}.
The inhomogeneous version of the theory \cite{Han:2015tga} was also constructed and was shown to support the same type of solutions \cite{ImCSH}. 
Although the existence of topological vortices was proved in \cite{Han:2015tga}, such proofs are still lacking for inhomogeneous vacuum solution and nontopological solitons and vortices.
It would be worth studying the issue in the future.

\section*{Acknowledgement}

We thank to H. W. Song and O-K. Kwon for valuable discussions.
This work was supported by the National Research Foundation of Korea(NRF) grant with grant number NRF-2022R1F1A1074051 (C.K.), NRF-2022R1F1A1073053 (Y.K.).

\end{document}